\documentclass[a4paper,UKenglish,autoref,thm-restate,numberwithinsect]{lipics-v2021}

\usepackage{amsmath,amssymb,amsthm}
\usepackage{mathabx}
\usepackage{mathtools}
\usepackage{thmtools}
\usepackage{algorithm}
\usepackage[noend]{algpseudocode}
\usepackage{enumitem}
\usepackage{textgreek}
\usepackage{algorithm}
\usepackage[noend]{algpseudocode}
\usepackage{silence}
\usepackage{hyperref}
\usepackage{cleveref}
\ErrorsOff*

\title{On Effective Banach-Mazur Games and an application to the Poincar\'e Recurrence Theorem for Category} 

\titlerunning{Effective Banach-Mazur Games and Poincar\'e Recurrence} 

\author{Prajval Koul}{Department of Computer Science and Engineering, Indian Institute of Technology Kanpur, India. \and \url{https://prajvalkoul.github.io/}}{prajvalk21@iitk.ac.in}{https://orcid.org/0000-0002-8879-8330}{}

\author{Satyadev Nandakumar}{Department of Computer Science and Engineering, Indian Institute of Technology Kanpur, Kanpur, Uttar Pradesh, India. \and \url{https://www.cse.iitk.ac.in/users/satyadev/} }{satyadev@cse.iitk.ac.in}{https://orcid.org/0000-0002-0214-0598}{}

\authorrunning{Prajval Koul and Satyadev Nandakumar}

\Copyright{Prajval Koul and Satyadev Nandakumar}

\ccsdesc[100]{Theory of computation}

\keywords{Recurrence, Topology, Category, Computable Analysis, Computable Toplogy, Dynamical Systems.}

\category{Invited Paper}

\relatedversion{}

\nolinenumbers %uncomment to disable line numbering

\EventEditors{Meena Mahajan, Florin Manea, Annabelle McIver, and Nguy\~{\^{e}}n Kim Th\'{\u{a}}ng}
\EventNoEds{4}
\EventLongTitle{43rd International Symposium on Theoretical Aspects of Computer Science (STACS 2026)}
\EventShortTitle{STACS 2026}
\EventAcronym{STACS}
\EventYear{2026}
\EventDate{March 9--September 13, 2026}
\EventLocation{Grenoble, France}
\EventLogo{}
\SeriesVolume{364}
\ArticleNo{5}
\newcommand{\R}{\mathbb{R}}
\newcommand{\Q}{\mathbb{Q}}
\newcommand{\N}{\mathbb{N}}
\newcommand{\X}{\mathcal{X}}

\begin{document}

\maketitle

\begin{abstract}
  The classical Banach-Mazur game characterizes sets of first category
  in a topological space. In this work, we show that an effectivized
  version of the game yields a characterization of sets of effective
  first category. Using this, we provide a game-theoretic proof of an
  effective theorem in dynamical systems, namely the category version
  of Poincar\'{e} Recurrence. The Poincar\'e Recurrence Theorem for
  category states that for a homeomorphism without open wandering
  sets, the set of non recurrent points forms a first category
  (meager) set. As an application of the effectivization of the
  Banach-Mazur game, we show that such a result holds true in
  effective settings as well.
\end{abstract}

\section{Introduction}
Existential arguments in classical mathematics often rely on the axiom
of choice, or its equivalent formulations like Zorn's lemma or the
Hausdorff maximal principle. Two major approaches in mathematics to
proving the existence of objects are probability and Baire category,
both of which abstractly study the ``size'' of a set of objects with
some property. Abstractly, if we are able to show that the ``size'' of
the set of objects is large, then this provides an indirect proof that
such an object must exist. In combinatorics, the \emph{probabilistic
  method} is a highly successful tool whereby complicated objects can
be shown to exist without necessarily providing a way to construct a
single concrete instance. Important combinatorial concepts like
expander graphs were initially shown to exist using such indirect
methods \cite{Pinsker1973} before explicit constructions were
obtained. The realm of algorithms and computability theory often
involves extracting the ``effective content'' of these theorems -
trying to make explicit the algorithmic content of these theorems,
insisting on explicit and efficient constructions of the objects.
These efforts often involve entirely new proofs of the classical
result.

In this work, we study the major tool in topology which is widely used
in analysis and topology to study the ``size'' of a class of sets,
namely Baire category. A set is small in this sense if it is
topologically meager (of first category). The Banach-Mazur game is a
two-player game where players take turns selecting from a class of
sets, and the outcome of the game characterizes sets of first
category. This game is one of the problems (problem number 43) in the
famous Scottish Book \cite{Mauldin2015}, the record of the
mathematical problems discussed in the Scottish Caf\'{e} in the city
of Lw\'{o}w, Poland (now Lviv, Ukraine) during the 1930s. Mazur
proposed the game in the Euclidean setting, and established one
direction. Banach proved the converse \cite{Banach1930} (see also
Mycielski et. al. \cite{Mycielski1956}). Kuratowski
\cite{Kuratowski1948} and Oxtoby\cite{Oxtoby1957} generalized this
game, and we study the effectivizations of these settings. Baire
Category has been previously studied in computability-theoretic
settings by Lisagor \cite{Lisagor1981}, and complexity theoretic
settings by Lutz \cite{Lutz1987}, \cite{Lutz1992}, Fenner
\cite{Fenner1995} and Breutzmann, Juedes and Lutz
\cite{BreutzmannJuedesLutz2001}.

% Lisagor1981, Lutz1992, BreutzmannJuedesLutz2001, Fenner1995,
% Lutz1987

Our first major result in this work is to provide an effective,
game-theoretic characterization of sets of the first (Baire) category,
also known as meager sets, using an effectivization of the classical
Banach-Mazur game. We then show two an easy application of this
effective framework - to show that the set of Liouville numbers is
effectively co-meager. % , and second, to prove an effective Banach
% category theorem

We further use this effective Banach-Mazur game to prove the effective
version of a fundamental result in dynamical systems, namely the
\emph{topological} version of the Poincar\'{e} recurrence theorem,
established first in the probabilistic setting by Poincar\'{e}
\cite{Poincare1899}, and generalized to topological settings initially
by Birkhoff \cite{Birkhoff1927} and by Hilmy \cite{Hilmy1940}. We show
that in every computable dynamical system, the set of non-recurrent
points is of effective first category (meager). Our proof is
game-theoretic, giving a computably enumerable winning strategy for
one of the players to win on the set of recurrent points.

In \cref{sec:prerequisites}, we recall and introduce several
topological notions, both classical and effective. In
\cref{sec:efftopo}, we discuss in detail the effective notions like
effective dense sets, computable topological spaces etc., upon which
the subsequent sections are built. In \cref{sec:effbmgame}, we
introduce the effective version of the Banach-Mazur game. There are
two versions of this effectivization, based on the nature of
meagerness of the set under consideration. Towards the end, we discuss
the effectivization of category in the Poincar\'e Recurrence Theorem
for bounded open regions of $\mathbb{R}^n$.

\section{Preliminaries}\label{sec:prerequisites}

This section consists of the required definitions and some basic results which we use in our work. We denote the binary alphabet by $\Sigma=\{0,1\}$. The set of finite binary strings is denoted by $\Sigma^*$, and the set of infinite binary sequences by $\Sigma^\infty$. The empty string is denoted by $\lambda$. The length of a finite string $w\in\Sigma^*$ is denoted by $\ell(w)$. For $x,y\in\Sigma^*$, $x$ being a prefix of $y$ is denoted as $x\sqsubseteq y$. The concatenation of two strings $x,y\in\Sigma^*$ is denoted by $x^\frown y\in\Sigma^*$. 

The set of rationals and reals is denoted by the usual symbols $\mathbb{Q}$ and $\mathbb{R}$ respectively. The set of natural numbers is denoted by $\mathbb{N}$. We assume a binary encoding $e:\mathbb{Q}\to\Sigma^*$ of the set of rationals. The complement of a set $A$ is denoted as $A^c$. For a set $X$, its power set is denoted by $\mathcal{P}(X)$. For a metric space $(X,\rho)$, the diameter of a set $A\subseteq X$ is denoted by $diam\ A=\sup_{x,y\in A}\rho(x,y)$. The disjoint union of two sets $A$ and $B$ is denoted by $A\coprod B$. 

The domain of a function $f:A\to B$ is denoted by $dom(f)=A$. A partial computable function $f$ from a countable set $A$ to a set $B$, denoted $f:A\dashrightarrow B$ is a function which is computable by a Turing Machine. A partial computable function is also referred to as a \textit{computably enumerable} (c.e.) function. Such function may be defined only on a subset of $A$. A total computable function $g:A \to B$ is a partial computable function whose domain is A. A computable enumeration is a partial surjection with domain $\N$. It will also be convenient to represent elements using strings. A partial computable surjection $f: \Sigma^* \dashrightarrow B$ is also called a \emph{representation} of elements in $B$. 

\subsection{Topology}

We outline the basic notions in topology which we require in our work. For a detailed exposition of these concepts, the reader may refer to the book on general topology by Engelking \cite{Engelking1989}. Briefly, a topological space $(X,\tau)$ is a space $X$ together with a class of sets $\tau$ called open sets. A class $\mathcal{B}$ of open sets is said to be a \emph{basis} for the topology on $X$ if every non-empty open set can be expressed as an arbitrary union of members of the basis. Any point in $X$ is an element of some basis set, and for every $A$, $B$ $\in \mathcal{B}$, there is a $C \in \mathcal{B}$ such that $C \subseteq A \cap B$ (see, for example, Engelking \cite{Engelking1989}, p. 12). A \emph{closed set} is the complement of an open set. There are sets which are neither open nor closed. In certain topologies, there are also sets which are \emph{both} open and closed, called clopen sets. The \emph{closure} of a set $A$, denoted $\overline{A}$, is the smallest closed set containing $A$. The \textit{interior} of a set $A$, denoted by $A^o$, is the largest open set that $A$ contains. 

A set $A$ is called \emph{dense} if its closure is $X$. A set $A$ is dense in an arbitrary open set $G$ if $(\overline{A \cap G}) = \overline{G}$. Every dense set is dense in every open set in the topology (see Engelking \cite{Engelking1989}, p.25). A set $A$ is \emph{nowhere dense} if every open set $B_1$ contains an open subset $B_2$ such that $B_2 \cap A = \emptyset$. 

A \emph{meager} set, or a set \emph{of first category}, is one which is a countable union of nowhere dense sets. The complement of a meager set is said to be \emph{co-meager}, or \emph{residual}. A set that is not meager is said to be \emph{of second category} (see Oxtoby \cite{Oxtoby1980}, p.40). As in the case of open and closed sets, there are sets which are neither meager, nor co-meager. In certain topologies, a set can be meager as well as co-meager. 

Note that meager sets can be dense - for example, the set of rationals are dense in $\mathbb{R}$, and they can clearly be expressed as a countable union of singleton sets which are nowhere dense.

\section{Effective Topological Spaces}\label{sec:efftopo}
In this section, we define the effective topological notions we require in later sections. In the most general setting, we work with computable $T_0$ spaces defined by Grubba, Schr\"oder and Weihrauch \cite{Grubba2007}. Classically, a $T_0$ space is a topological space $(X,\tau)$ such that for every pair of distinct $x_1$, $x_2$ $\in X$, there is an open set that contains one and only one of these points. 

\begin{definition}[Representation of a countable class \cite{Grubba2007}]
	For a countable class $\mathcal{U}$, the partial computable 	surjection $\nu:\Sigma^*\dashrightarrow\mathcal{U}$, where $(\forall w\in dom(\nu))(\nu(w)\ne\varnothing)$, is said to be the representation of $\mathcal{U}$. 
\end{definition}

It should be noted that the domain of $\nu$ is a c.e. set. $\nu^{-1}(U)$ refers to a \emph{name} of the set $U\in\mathcal{U}$. Computationally, it is necessary to ``name'' the basic open sets, hence we work with a countable base, necessitating a second countable $T_0$ space.

\begin{definition}[Computable $T_0$ Space \cite{Grubba2007}]
	A \emph{computable $T_0$ space} is a tuple $\X=(X,\tau,\beta,\nu)$ such that $(X,\tau)$ is a second countable $T_0$ space, and $\beta$ is a countable basis for the space where $\nu: \Sigma^* \dashrightarrow \beta$ is a representation of the countable basis $\beta$, such that the following condition holds. 
	\begin{align}
		\label{label_intersection}
		\left(\forall u, v \in \text{dom}(\nu)\right)\quad \left(\nu(u)\cap\nu(v) =
		\bigcup \{\nu(w) \mid (u,v,w) \in B\} \text{ for a c.e. set } B\right),
	\end{align}
	and $U \ne \emptyset$ for $U \in \beta$.
\end{definition}

The members of the basis of a computable topological space are also called \textit{basic open sets} of the space. The condition (\ref{label_intersection}) ensures that the intersection of sets in the basis is c.e.. The property ``$U\ne\emptyset$ for $U\in\beta$'' excludes the empty set from the basis $\beta$. This has the following consequence. 

\begin{lemma}[\cite{Grubba2007}]
	For a \emph{computable $T_0$ space} $(X,\tau,\beta,\nu)$, the relation $\{(u,v)\in\Sigma^*\times\Sigma^*:\nu(u)\cap\nu(v)\ne\emptyset\}$ is c.e.. 
\end{lemma}

The above lemma follows from the fact that the intersection is non-empty if, and only if, there is a $w\in\Sigma^*$ such that $\nu(w)\subseteq\nu(u)\cap\nu(v)$, which can be discovered by a standard dovetailing argument. 

Observe that $\mathbb{R}$ with the standard topology is an example of a computable $T_0$ space. 

We now define the notions in effective topology which we use in our work. In this section, we define only those notions which we require throughout our discussion. Later, we have results which hold in computable metric spaces, and computable dynamical systems. We introduce those notions in the relevant sections. 

\begin{definition}[c.e. open sets and co-c.e. closed sets]
	A set $U$ is said to be a \emph{computably enumerable open} (c.e. open) set if it can be written as a computably enumerable union of the basic open sets of the space. A set $F$ is said to be a \emph{co-c.e. closed set} if $F^c$ is c.e. open. 
\end{definition}

Note that the terms \textit{effective open} and \textit{c.e. open} are used interchangeably throughout the sections. 

It follows easily from the definition that a c.e. union of c.e. open sets is c.e. By taking complements, a c.e. intersection of co-c.e. closed sets is co-c.e. closed. 

We now discuss notions of meagerness of sets in effective spaces. These are the central notions in our work. 

\begin{definition}[Effective nowhere dense set]
	Let $X=(X,\tau,\beta,\nu)$ be a computable $T_0$ space. A set 	$A\subseteq X$ is said to be effectively nowhere dense in $X$ if there exists a computable function $f:\Sigma^*\to\Sigma^*$ such that for $w \in \text{dom}(\nu)$, we have 
	\begin{align}
		\nu(f(w))\subseteq(\nu(w)\setminus A)^o.
	\end{align}
\end{definition}

The above definition allows for uncountable sets to be effectively nowhere dense. In the standard topology of $\R$, the set of natural numbers is an effective nowhere dense set. In the same topology, it is possible to show that the Cantor set is an effective nowhere dense set. 

The following are some useful results regarding effective nowhere dense sets. 

\begin{lemma}
	\label{lem:effective_nowhere_dense}
	Let $\X=(X,\tau,\beta,\nu)$ be a computable $T_0$ space. Then the 	following hold. 
	\begin{enumerate}[noitemsep,nolistsep]
		\item Any finite intersection of effective nowhere dense sets is 	effectively nowhere dense. 
		\item The closure of an effective nowhere dense set is effectively 	nowhere dense. 
	\end{enumerate}
\end{lemma}

\begin{proof}
	\begin{enumerate}[noitemsep,nolistsep]
		\item Suppose $A_1$, $\dots$, $A_n$ are effectively nowhere dense. Let $w \in \Sigma^*$ be arbitrary. Since the $A_i$s are nowhere dense, there are c.e. open sets $U_1$, $\dots$, $U_n$ such that for each $1 \le i \le n$, $U_i \subseteq \nu(w)$ and $U_i \cap A_i = \varnothing$. Then it follows that $\bigcap_{i=1}^n U_i \cap A_i = \left(\bigcap_{i=1}^n U_i\right)\cap \left(\bigcap_{i=1}^n A_i\right) = \varnothing$. Since the finite intersection of c.e. open sets is open, $\left(\bigcap_{i=1}^n U_i\right)$ is a c.e. open set which is a subset of $\nu(w)$ such that its intersection with $\left(\bigcap_{i=1}^n U_i\right)$ is empty. This procedure is uniform in $w$, hence $\left(\bigcap_{i=1}^n A_i\right)$ is an effective nowhere dense set. 
		\item Let $A$ be effectively nowhere dense set and $w\in\Sigma^*$. Let $B \subseteq \nu(w)$ be a c.e. open set which is contained in $A^c$. Since $B$ is open, it is also contained in $(\overline{A})^c$. 
	\end{enumerate}
\end{proof}

\begin{lemma}\label{lem:cedensecomp}
	The complement of a dense c.e. open set is effective nowhere dense. 
\end{lemma}
\begin{proof}
	Let $A\subseteq X$ be a dense c.e. open set. We can express $A=\bigcup_{i\in\mathcal{I}}A_i$, where $\{A_i\}_{i\in\mathcal{I}}\subseteq\mathcal{P}(X)$ is a sequence of basic open sets in $X$. Since $A$ is dense in $X$, for every basic open set $U\subseteq X$, $A\cap U\ne\varnothing$. Also, since $A$ is open in $X$, $A\cap U$ is also open in $X$. Thus, we can enumerate a basic open set $V_A\subseteq X$ such that $V_A\subsetneq A\cap U$.	Therefore, $V_A\subseteq U$ and $V_A\cap\overline{A^c}=\varnothing$. Since	this holds for every such $U$, by definition, $A^c$ is effective	nowhere dense in $X$. 
\end{proof}

\begin{lemma}\label{lem:nowhere_dense_complement}
	The complement of an effective nowhere dense set contains a dense c.e.  open set. 
\end{lemma}
\begin{proof}
	Let $A\subseteq X$ be an effective nowhere dense set, and let $f:\Sigma^* \to \Sigma^*$ be a computable function witnessing that $A$ is nowhere dense. By definition, we know that for any $w \in \Sigma^*$, the non-empty set $\nu(f(w))$ is a non-empty basic open set contained in $(\nu(w)\setminus A)^o$. Since this interior is non-empty for every $w \in \Sigma^*$, it follows that $\bigcup_{w \in \Sigma^*} \nu(f(w))$ is dense. Further, we have	that $\bigcup_{w \in \Sigma^*} \nu(f(w))$ is a computably enumerable union of non-empty open sets, hence is c.e. open. Thus $A^c$ contains $\bigcup_{w \in \Sigma^*} \nu(f(w))$, a dense, c.e. open set in $X$. 
\end{proof}

\begin{definition}[Effective First Category Set]
	A set is said to be \emph{of effective first category}, if it can be represented as a c.e. union of effective nowhere dense sets. 
\end{definition}

A set of effective first category is also called an \textit{effective meager set}. Sets which are not of effective first category are called sets of \textit{effective second category}.

\section{Effective Banach-Mazur Games}\label{sec:effbmgame}
We now describe the classical Banach-Mazur game \cite{Oxtoby1957}. The goal of the game is to show that a particular set is of first category. The original game was defined on the real line and later generalized. We mention the general setting considered by Oxtoby \cite{Oxtoby1957}. Two players, denoted $P_1$ and $P_2$, take turns picking sets, in order to show that a designated set is of first category. 

Consider the parent space $(X,\tau,\beta,\nu)$. The game is denoted as $BM \langle M,C \rangle$, where $M$ and $C$ are disjoint, and $M\cup C=X$. There are 2 players, denoted $P_1$ and $P_2$. $M$ is the the \textit{target set} for $P_1$, and $C$ for $P_2$. The game specifies a class $\mathcal{G}$ of sets with non-empty interior and such that every non-empty open set contains some set from $\mathcal{G}$. At every turn, the players are supposed to choose sets from this class. The game starts with $P_1$ choosing a set $G_1\in\mathcal{G}$, followed by $P_2$ choosing a set $G_2\subseteq G_1$, $G_2\in\mathcal{G}$, and so on. At the $n^{th}$ move of the corresponding player, $P_1$ chooses a set $G_{2n-1}\subseteq G_{2n-2}$, $G_{2n-1}\in\mathcal{G}$ and $P_2$ chooses a set $G_{2n}\subseteq G_{2n-1}$, $G_{2n}\in\mathcal{G}$. $P_1$ wins the game if $M\cap\bigcap_{n\geq1}G_n\ne\varnothing$. Else, $P_2$ wins \cite{Oxtoby1957}. 

There are two distinct results about the game. The first, more general, version shows that the set $M$ is of first category if, and only if, $P_2$ has a winning strategy. We introduce here the effective version of this game.

\subsection{The Effective Banach-Mazur Game (Version 1)}
We introduce the relevant notions for the effectivization as and when they are required. Along the way, we also justify the necessity for using these notions over the ones previously defined. 

\begin{definition}[Strongly Computable $T_0$ Space]
	A \emph{strongly computable $T_0$ space} $(X,\tau,\beta,\nu)$ is a computable $T_0$ space such that for all $u,v\in dom(\nu)$, the operation $\nu(u)\cap\nu(v)=\varnothing$ is decidable. 
\end{definition}

In a strongly computable $T_0$ space, the disjointness, inclusion, and intersection of basic open sets in the respective space become computably enumerable. 

Now, to effectivize the game, we take a strongly computable space $(X,\tau,\beta,\nu)$ as the parent space. We also impose computational restrictions on one of the players. In the first version, we assume $P_1$ to have unbounded computational resources while picking from the collection $\mathcal{G}$. $P_2$, on the other hand, can only have an effective strategy. An effective strategy entails the computation of the response set in an unbounded finite time via a computable function. 

\begin{definition}[Effective strategy for $P_2$]
	An effective strategy for the second player is denoted by $\mathcal{G}^{(2)} = \{G_{2k}:G_{2k}=\nu(f_k\left(\nu^{-1}(G_1),\nu^{-1}(G_2),\ldots,\nu^{-1}(G_{2k-1})\right))\}_{k\geq1}\subseteq \mathcal{G}$, where $\{f_n\mid f_n:(\Sigma^*)^{2n-1}\to\Sigma^*\}_{n\geq1}$ is a uniformly computable (in $n$) sequence of computable choice functions, where, for each $i\in\N$, $G_i\in\mathcal{G}^{(2)}$. 
\end{definition}

Uniform computable here refers to a single Turing functional computing every bit of the output. Note that the family $\{f_n\}_{n\geq1}$ is a uniformly computable family of functions. By definition, each $f_i\in\{f_n\}_{n\geq1}$ computes a basic open set. For simplicity, we can identify each $i\in\mathbb{N}$ with the basic open set $G_i\in\mathcal{G}^{(2)}$ that it represents. 

\textbf{Note.} In the case of a strongly computable $T_0$ space, $\mathcal{G}$ consists of basic open sets of the space. At any given stage, the possible choices (from the class $\mathcal{G}$) for any player can be assumed to be computably enumerable. In other words, at stage $k$ of the game, the class of sets from which $P_2$, for instance, picks the set to be played, is a c.e. family of basic open sets. 

Since $P_2$ can only play an effective strategy, operations like unions and intersections of basic open sets of the space are permitted (by virtue of the parent space being a strongly computable $T_0$ space). 

\subsubsection{Characterization of Effective First Category Sets}

The classical Banach-Mazur game yields a characterization of sets of first category \cite{Oxtoby1957}. Here we show that the effective version of the game yields a characterization of effective first category sets. 

The classical proofs (for example, see Oxtoby \cite{Oxtoby1957} and Oxtoby \cite{Oxtoby1980}) use Zorn's lemma to establish one of the implications. Since we deal with effective strategies and effective first category sets, we cannot appeal to existential arguments. One of the important contributions of the following proof is to provide an explicitly constructive argument, avoiding appeals to the axiom of choice, or to its equivalent formulations like Zorn's lemma or the Hausdorff maximum principle. 

\begin{theorem}\label{thm:effbm1}
	In a strongly computable $T_0$ space $(X,\tau,\beta,\nu)$ with $M\coprod C=X$, the Banach-Mazur game $BM\langle M,C\rangle$ has an effective winning strategy for $P_2$ if, and only if, $M$ is an effective first category set in $X$. 
\end{theorem}
\begin{proof}
	Let $(X,\tau,\beta,\nu)$ be a strongly computable $T_0$ space, and $M\subseteq X$ be an effective first category set. Thus,	$M=\bigcup_{n\geq1}M_n$, a c.e. union of a sequence $\{M_n\}_{n\geq1}$ of effective nowhere dense sets in $X$. Both players must choose from the class $\mathcal{G}$ of basic open sets.	We now describe an effective winning strategy for $P_2$. 
	
	At stage $k$ of the game, let $P_1$'s choice be $G_{2k-1}\in\mathcal{G}$, where $G_{2k-1}$ is a basic open set. Consider the set $G_{2k-1}\setminus\overline{M_k}$. We show that it is a c.e. open set. By \cref{lem:effective_nowhere_dense}, we have	that $\overline{M_k}$ is an effective nowhere dense set. Its complement, therefore, contains a dense c.e. open set. The intersection of two c.e. open sets is c.e. open, hence $G_{2k-1}\setminus\overline{M_k}$ is a c.e. open set, which is a c.e. union of basic open sets. Computably enumerate the basic open sets which constitute $G_{2k-1} \setminus\overline{M_k}$, and let $G$ be the first basic open set in this enumeration. Clearly, $G \in \mathcal{G}$. $P_2$ plays the set $G_{2k}=G$. We now show that this is a winning strategy for $P_2$. 
	
	Now,
	\begin{align*}
		\bigcap_{n\geq1}G_n &=\bigcap_{k\ge1}G_{2k-1}\cap G_{2k}\\
		&\subseteq \bigcap_{k\ge1}G_{2k-1}\cap\left(G_{2k-1}\setminus \overline{M_k}\right)&\\
		&=\bigcap_{k\ge1} G_{2k-1}\mathbin{\big\backslash} \bigcup_{k\geq1}
		\overline{M_k}.
	\end{align*}
	Hence,
	\begin{align*}
		M\bigcap\left(\bigcap_{n\geq 1}G_n\right)
		&\subseteq M\bigcap \left(\bigcap_{n \ge 1} G_n \cap \overline{M_n}^c\right)\\
		&\subseteq M\bigcap \left(\bigcap_{n \ge 1} G_n
		\cap M^c\right)\\
		&=\varnothing,
	\end{align*}
	where the second subset relationship follows since $M=\bigcup_{n\in\N}M_n \subseteq \bigcup_{n\in\N} \overline{M_n}$,	implying that $M^c \supseteq \bigcap_{n\in\N} \overline{M_n}^c$. Thus, if $M$ is an effective first category set, this is an effective winning strategy for $P_2$. 
	
	Conversely, let $P_2$ have an effective winning strategy denoted by $\mathcal{G}^{(2)} =\{G_{2k}:G_{2k} = \nu(f_k\left(\nu^{-1}(G_1),\nu^{-1}(G_2), \ldots, \nu^{-1}(G_{2k-1})\right))\}_{k\geq1}\subseteq\mathcal{G}$ where $\forall j\geq1$, $G_{2j}\subseteq G_{2j-1}$ is a member of the class $\mathcal{G}$. At any stage $n$, consider the sequence of sets $G_1 \supseteq G_2 \supseteq \dots \supseteq G_{2n}$, where for $i\in\{1,2,\ldots,n\}$, we have $\nu(f_i\left(\nu^{-1}(G_1),\ldots,\nu^{-1}(G_{2i-1})\right))=G_{2i}$ according to the strategy of $P_2$. We call this descending sequence of sets, an $n$-chain (Note that each instance of the game, corresponding to the choices made by $P_1$ and $P_2$, leads to a distinct chain). The set $G_{2n}$ is designated as the top of the chain. An $(n+k)$-chain is a \emph{continuation} of an $n$-chain if the first $2n$ sets in this chain are the same as in the $n$-chain. Then continuation forms a partial ordering among the collections of all possible chains. Note also that since $\mathcal{G}$ is computably enumerable, the collection of $n$-chains is c.e. uniformly in $n$. 
	
	For $n\ge1$, we now construct a maximal c.e. family $\mathcal{H}_n$ of basic open sets such that their union is dense in $X$. Let $\{C^{(n)}_i: i\in\N\}$ be the computable enumeration of $n$-chains, where each $C^{(n)}_i$ consists of $2n$ nested basic open sets (denoting a possible play of the game up to stage $n$). Initially pick $G_{1}=\nu(f_n(\nu^{-1}(\{C^{(n)}_1\})))$ (by slight abuse of notation) and add it to $\mathcal{H}_n$. At any stage $k> 1$ of construction of $\mathcal{H}_n$, suppose $\mathcal{H}_n$ be a finite collection of basic open sets which have been selected using the chains $C^{(n)}_1$, $\dots$, $C^{(n)}_{k-1}$. Now, from chains $C^{(n)}_j$, $j\geq k$, from among the topmost basic open sets $G_{2n,j}$ of each chain $C_j^{(n)}$, pick the set $G_{2n,j}$ with the least $j$, which is disjoint from any of the sets currently in $\mathcal{H}_n$. Add this set into the collection $\mathcal{H}_n$. Since there are only at most $k$ sets in $\mathcal{H}_n$ up to stage $k$, and disjointness of basic open sets is decidable in a strongly computable $T_0$ space, this step is computable. 
	
	By construction, $\mathcal{H}_n$ is a maximal family of disjoint collection of basic open sets within $X$. By maximality, $\bigcup\mathcal{H}_n$ is a dense open set. Since for every basic open set $B$, we can computably enumerate a member $\bigcup\mathcal{H}_n$ which is contained in $B$, it follows that $\bigcup\mathcal{H}_n$ is an \emph{effectively} dense c.e. open set. 
	
	Consider the set $G=\bigcap_{n\geq1}\bigcup\mathcal{H}_n$. Since $f_i$ is part of the winning strategy for $P_2$, we have $\bigcap_{n\geq1}\bigcup\mathcal{H}_n \subseteq C$. This is a c.e. intersection of dense c.e. open sets. Hence $G^c\supseteq M$ is of effective first category in $X$. Hence $M$ is of effective first category. 
\end{proof}

\textbf{Remarks.} 
\begin{itemize}
	\item The parent space is required to be a strongly computable $T_0$ space. Working with just a computable $T_0$ space is not sufficient, since we need to check for disjointness. 
	\item The class $\mathcal{G}$ of playable sets is essentially the basis of the computable $T_0$ space under consideration. Though this may seem restrictive, it leads to the characterization of effective first category sets. Note that even the choices in the classical game are restricted. For instance, the initial version of the game (mentioned in	Oxtoby \cite{Oxtoby1980}) requires players to pick a closed interval of the real line, not any arbitrary closed set. 
\end{itemize}

\subsection{The Effective Banach-Mazur Game (Version 2)}
We saw that the above game acts as a characterization of effective first category sets. One might wonder under what conditions could $P_1$ win. The following theorem establishes that if the complement of $P_1$'s target set is of effective first category at some point $x$ of the parent space, then $P_1$ has a winning strategy. 

\begin{definition}[Effective first category (set) at a point]
	For a strongly computable $T_0$ space $(X,\tau,\beta,\nu)$, the set $A\subseteq X$ is said to be of effective first category at a point $x\in X$ if there is some non-empty neighborhood $N_x\subseteq X$ of $x$ such that $N_x\cap A$ is of effective first category in $X$. 
\end{definition}

In this version of the game, we assume $P_2$ to have unbounded computational resources while picking from the collection $\mathcal{G}$. $P_1$, on the other hand, can only have a computable strategy. The respective winning criterion remains the same as before. 

\begin{definition}[Effective strategy for $P_1$]
	An effective strategy for the first player is denoted by $\mathcal{G}^{(1)} = \{G_{2k-1}:G_{2k-1} = \nu(f_k\left(\nu^{-1}(G_1),\nu^{-1}(G_2),\ldots,\nu^{-1}(G_{2k-2})\right))\}_{k\geq1}\subseteq\mathcal{G}$, $\{f_n\mid f_n:(\Sigma^*)^{2n-2}\to\Sigma^*\}_{n\geq1}$ is a c.e. sequence of computable choice functions uniform over $n$, with each $n\in\mathbb{N}$ corresponding to the basic open set $G_n\in\mathcal{G}^{(1)}$. 
\end{definition}

Recall that a set is of first category at a point if it is of first category at an open neighborhood of that point. We remark that this condition is much weaker than being an effective first category set, which was the requirement in the last game. Hence, we need the parent space, in addition to being strongly computable and $T_0$, to have some more properties. We can not work with simply a strongly computable $T_0$ space here, unlike earlier, for there is no notion of convergence at a point, something that qualifies as a winning criterion for $P_1$. We need the parent space to at least be equipped with a metric, owing to which we can quantify the convergence at every stage of the game. We also require the space to be complete. 

Recall the notion of a computable metric space. The following allied notions are required in the current version of the game. 
\begin{definition}[Computable Metric Space]
	A computable metric space is a tuple $(X,\rho,W,\nu)$, where $(X,\rho)$ is a metric space, $\nu$ is a representation of the	parent space, and $W=\{w_i\}_{i\geq1}\subseteq\Sigma^*$ is a sequence of points with the property that $\{\nu(w_i)\}_{i\geq1}$ is dense in $(X,\rho)$, such that for all $i,j\in\mathbb{N}$, $\rho(\nu(w_i),\nu(w_j))$ is computable. 
\end{definition}
\begin{definition}[Complete Metric Space]
	A metric space $(X,\rho)$ is said to be complete if every Cauchy sequence converges. In other words, for a Cauchy	sequence $\{x_n\}_{n\geq1}\subseteq(X,\rho)$, there exists an $x\in X$ such that $x_n\rightarrow x$. 
\end{definition}
\begin{lemma}[Cantor]\label{lem:cantor}
	Let $(X,\rho)$ be a complete metric space. For a decreasing sequence $F_1\supseteq F_2\supseteq \cdots$ of non empty closed subsets of $X$ such that $diam\ F_n\rightarrow0$, there is an $x\in X$ such that $\bigcap_{n\geq1}F_n=\{x\}$. 
\end{lemma}

For computable complete metric spaces, Yasugi, Mori, and Tsujii \cite{YMT99} and independently Brattka \cite{Brattka2001} have effectivized the Baire category theorem. 

We also require the notion of convergence of sets to an interior point. The following property of a computable metric space is useful in this regard.
\begin{lemma}\label{lem:propersubset}
	Let $(X,\rho,W,\nu)$ be a computable metric space. Then, for every non-empty c.e. open set $U\subseteq X$, there exists a non-empty basic open set $V\subseteq X$ such that $\overline{V}\subsetneq U$. 
\end{lemma}
\begin{proof}
	Let $U\subseteq X$ be a c.e. union of basic open balls $B_\rho(\alpha_i,r_i)$, where $B_\rho(\alpha,r)=\{x\in X:\rho(\alpha,x)<r\}$ and $r_i\in\mathbb{Q}$, $\alpha_i\in X$, $i\in\mathbb{N}$. Then $\overline{B_\rho(\alpha_1,\frac{r_1}{2})}=\{x\in X:\rho(\alpha_1,x)\leq\frac{r_1}{2}\}$ is such that $\varnothing\ne\{\alpha_i\} \subseteq \overline{B_\rho(\alpha_1,\frac{r_1}{2})}\subsetneq        B_\rho(\alpha_1,r_1)$. This is the required set. 
\end{proof}
Now, we are ready to discuss a situation wherein $P_1$ wins the effective Banach-Mazur game. 
\begin{theorem}\label{thm:effbm2}
	For a complete computable metric space $(X,\rho,Y,\nu)$ with $M\coprod C=X$, the Banach-Mazur game $BM\langle C,M\rangle$ has an effective winning strategy for $P_1$ if, and only if, $M$ is of effective first category at some point in $X$.
\end{theorem}
The reader should be careful and especially note the changed labels in the theorem statement (and the upcoming proof). The labels indicate the nature of the respective sets. 
\begin{proof}
	Let $(X,\rho,Y,\nu)$ be a complete computable metric space. Let $\mathcal{G}\subseteq\mathcal{P}(X)$ be defined as in the previous version of the game.
	
	Let $M$ be of effective first category at some point $z\in X$. Let $G\in\mathcal{G}$ be such that $z\in G$ and $G\cap M$ is of effective first category in $X$. Therefore, we can write $G\cap M=\bigcup_{n\geq1}M_n$, where  $\{M_n\}_{n\geq1}\subseteq\mathcal{P}(X)$ is a sequence of effective nowhere dense sets in $X$. $P_1$ begins by playing the set $G_1=G\cup M$. Recall that $P_1$'s choices can be computably enumerated. Let $P_2$'s response at stage $k-1$ of the game be $G_{2k-2}\in\mathcal{G}$. Then, towards the next stage of the game, $P_1$ picks, out of the enumeration, the first set $G'\in\mathcal{G}$ with $diam\ G'<\frac{1}{k}$ such that $\overline{G'}\subsetneq G_{2k-2}\setminus\overline{M_k}$. The existence and enumerability of such a set is ensured by \cref{lem:propersubset}. $P_1$ plays the set $G_{2k-1}=G'$. 
	
	Considering the above play of the game, we note that $\bigcap_{n\geq1}G_n\subseteq C$. Since $\overline{G_{2k-1}}\subseteq G_{2k-2}$, we get $\bigcap_{n\geq1}\overline{G_n}=\bigcap_{n\geq1}G_n$. With $diam\ G_{2n-1}<\frac{1}{n}$, by \cref{lem:cantor}, $\bigcap_{n\geq1}\overline{G_n}$ is a singleton, say $y\in C$. Hence $C\cap\bigcap_{n\geq1}G_n\neq\varnothing$, which shows that this is an effective winning strategy for $P_1$. 
	
	Conversely, let $P_1$ have an effective winning strategy $\mathcal{G}^{(1)}\subseteq\mathcal{F}$ denoted as $\mathcal{G}^{(1)}=\{G_{2k-1}:G_{2k-1}=\nu(f_k\left(\nu^{-1}(G_1),\nu^{-1}(G_2),\ldots,\nu^{-1}(G_{2k-2})\right))\}_{k\geq1}\subseteq\mathcal{G}$, where $\{f_n\mid f_n:(\Sigma^*)^{2n-2}\to\Sigma^*\}_{n\geq1}$ is a sequence of computable choice functions. This strategy intersects at a non-empty set. Therefore, there is a basic open set $U\subseteq\bigcap_{n\geq1}G_n$ lying in the intersection. The set $U$ depends on the moves of both the players, and hence, irrespective of $P_1$'s strategy, may be different for every instance of the game. This set contains a point $x\in C$. Now, with $G_1\in\mathcal{G}^{(1)}$ being the first set that $P_1$ plays, it suffices to show that $G_1\cap M$ is of effective first category in $X$. This is because if $M$ is of effective first category at some point $w\in X$, then there is some non-empty neighborhood $N_w\subseteq X$ of $w$ such that $N_w\cap M$ is of effective first category in $X$. We assert that $G_1$ is such a neighborhood. 
	
	Since $P_1$ plays $G_1$ as the first move, the current stage of the game is transformed into $BM\langle G_1\cap M,X\mathbin{\big\backslash}G_1\cap M\rangle$ with the original $P_2$ playing the first move. Since this game has a winning strategy for the current \textit{new} $P_2$, by \cref{thm:effbm1}, $G_1\cap M$ is an effective first category set. 
\end{proof}
At this point we remark that the proof of \cref{thm:effbm2}, as opposed to the proof of \cref{thm:effbm1}, specifically asks for $X$ to be a metric space, since we use the notion of a diminishing sequence of diameters of sets. This is what enables us to use Cantor's lemma.

\subsection{An Application: the complement of the set of Liouville
  numbers}
In this section, we discuss a direct application of the above effective
versions of the Banach-Mazur games.

Joseph Liouville was the first to demonstrate the existence of a
transcendental real number \cite{Liouville1844}, now named Liouville
constant. He showed that this constant has very good Diophantine
approximation, and numbers which have such close Diophantine
approximations are all transcendental. Numbers which obey this
property are now called Liouville numbers.

\begin{definition}[Liouville Number]
  A number $x\in\mathbb{R}\setminus\mathbb{Q}$ is said to be Liouville, if 
  \begin{align*}
    (\forall n\in\mathbb{N})(\exists p,q\in\mathbb{Z})
    \left(q>1 \wedge \left|x-\frac{p}{q}\right| <\frac{1}{q^n} \right). 
  \end{align*}
\end{definition}

Let $E$ denote the set of Liouville numbers. We can write $E$ as 
\begin{align*}
  E \quad=\quad
  \mathbb{R}\setminus\mathbb{Q} \cap
  \bigcap_{n\in\N}\bigcup_{\substack{p,q\in\mathbb{Z}\\q>1}}
  B_\rho\left(\frac{p}{q},\frac{1}{q^n}\right)
  \quad=\quad
  \mathbb{R}\setminus\mathbb{Q}\cap
  \bigcap_{n\in\N}\bigcup_{\substack{p,q\in\mathbb{Z}\\q>1}}
  \left(\frac{p}{q}-\frac{1}{q^n},\frac{p}{q}+\frac{1}{q^n}\right), 
\end{align*}
where $\rho$ is the usual Euclidean metric. The set $E$ has Lebesgue
measure 0 and Hausdorff dimension 0 \cite{Oxtoby1980}. Thus, $E$ is a
very small set in terms of measure and Hausdorff dimension. In fact,
its effective Hausdorff dimension is 0 \cite{Staiger2002}, even though
$E$ contains normal numbers, hence its finite-state dimension is 1
\cite{Kano1993,Nandakumar2016}.

It is natural to consider whether such a set $E$ is topologically
small. Surprisingly, it turns out to be otherwise. we observe that
$E^c$ is effectively meager, hence $E$ is effectively co-meager.

\begin{theorem}
	The set of non-Liouville numbers forms an effective first category set. 
\end{theorem}
\begin{proof}
The complement of the set of Liouville numbers can be written as
\begin{align*}
  E^c&=\mathbb{Q} \cup \bigcup_{n\in\N}
       \left(\bigcup_{\substack{p,q\in\mathbb{Z}\\q>1}}
  \left(\frac{p}{q}-\frac{1}{q^n},\frac{p}{q}+\frac{1}{q^n}\right)\right)^c.  
\end{align*} 
Observe that the set
$\bigcup_{\substack{p,q\in\mathbb{Z}\\q>1}}\left(\frac{p}{q}-\frac{1}{q^n},\frac{p}{q}+\frac{1}{q^n}\right)$ 
is dense in $\mathbb{R}$, since it contains all rationals. Moreover,
it is easily seen to be a c.e. open set. Hence, it is a dense c.e.
open set, and its complement is effective nowhere dense, by
\cref{lem:cedensecomp}.
%It should be noted that this set is meager. 

Now, consider the Banach-Mazur game $BM\langle E,E^c \rangle$.
Consider the following strategy for $P_2$. At stage $k$ of the game,
with $P_1$'s response being $G_{2k-1}\in\mathcal{G}$, $P_2$ plays the
set $G_{2k}\in\mathcal{G}$, where $G_{2k}$ is the first set in
$\mathcal{G}$ such that
\begin{align*}
  G_{2k}&\subseteq G_{2k-1}\mathbin{\big\backslash}\overline{\left(\bigcup_{\substack{p,q\in\mathbb{Z}\\q>1}}\left(\frac{p}{q}-\frac{1}{q^k},\frac{p}{q}+\frac{1}{q^k}\right)\right)^c}\\
	&=G_{2k-1}\cap\left(\overline{\left(\bigcup_{\substack{p,q\in\mathbb{Z}\\q>1}}\left(\frac{p}{q}-\frac{1}{q^k},\frac{p}{q}+\frac{1}{q^k}\right)\right)^c}\right)^c\\
	&=G_{2k-1}\cap\left(\bigcup_{\substack{p,q\in\mathbb{Z}\\q>1}}\left(\frac{p}{q}-\frac{1}{q^k},\frac{p}{q}+\frac{1}{q^k}\right)\right)^o\\
	&=G_{2k-1}\cap\bigcup_{\substack{p,q\in\mathbb{Z}\\q>1}}\left(\frac{p}{q}-\frac{1}{q^k},\frac{p}{q}+\frac{1}{q^k}\right)^o.
\end{align*}
Clearly
$\bigcup_{\substack{p,q\in\mathbb{Z}\\q>1}}\left(\frac{p}{q}-\frac{1}{q^n},\frac{p}{q}+\frac{1}{q^n}\right)^o$
is a non-empty c.e. open set in $\mathbb{R}$. Hence $G_{2k}$ is a
non-empty basic open set of the parent space. $P_2$ plays this set.

Now, observe that
\begin{align*}
  E^c\cap\bigcap G_k&=E^c\bigcap\left(\bigcap G_{2k-1}\cap G_{2k}\right)\\
                    &\subseteq E^c\bigcap\left(\bigcap_{k\in\N}G_{2k-1}\cap\left(G_{2k-1}\mathbin{\big\backslash}\overline{\left(\bigcup_{\substack{p,q\in\mathbb{Z}\\q>1}}\left(\frac{p}{q}-\frac{1}{q^k},\frac{p}{q}+\frac{1}{q^k}\right)\right)^c}\right)\right)\\
                    &=E^c\bigcap\left(\bigcap_{k\in\N}G_{2k-1}\mathbin{\big\backslash}\bigcup_{k\in\N}\overline{\left(\bigcup_{\substack{p,q\in\mathbb{Z}\\q>1}}\left(\frac{p}{q}-\frac{1}{q^k},\frac{p}{q}+\frac{1}{q^k}\right)\right)^c}\right)\\
                    &\subseteq E^c\bigcap\left(\bigcap_{k\in\N}G_{2k-1}\mathbin{\big\backslash}E^c\right)\\
                    &=\varnothing.
\end{align*}
Hence, the above strategy is an effective winning strategy for $P_2$
in the game $BM\langle E^c,\mathbb{R}\setminus E^c\rangle$. Therefore,
by \cref{thm:effbm1}, $E^c$ is an effective first category set.
\end{proof}

\section{Categorization of Sets of Non-recurrent Points}

In this section, we discuss a major application of our effectivization. The effective version of the Banach-Mazur game helps categorize the set of non-recurrent points of a dynamical system. We show that the set of non-recurrent points of any suitably effectivized dynamical system forms a set of effective first category, similar to the category version of the classical Poincar\'e recurrence theorem. 

%\subsection{Categorization of Sets of Non-recurrent Points}

The Poincar\'{e} recurrence theorem is a pioneering and fundamental result in the theory of dynamical systems \cite{Poincare1899}. It shows that in a deterministic dynamical system which is appropriately bounded, usually expressed in terms of a finite measure, or being topologically bounded, nearly all the points in phase space return infinitely often, arbitrarily close to their initial positions. This behavior prevents most points in the phase space from ``escaping to infinity'' (see, for example, Walters \cite{Walters1982} for the standard measure-theoretic version). This theorem was also influential in the history of physics. Physicists, starting with Boltzmann \cite{Boltzmann1896} and Zermelo \cite{Zermelo1896}, have studied its implication to the second law of thermodynamics.

First, we define the essential notions from dynamical systems which we require. The classical version below is quoted for measure as well as for category. The reader is referred to Oxtoby \cite{Oxtoby1980} for details. The effective measure theoretic Poincar\'{e} theorem is known - it follows from the effective Birkhoff ergodic theorem \cite{Vyugin1997}, \cite{Nandakumar2008}, \cite{Hoyrup2009b}, \cite{Vijayvargiya2014} and the effective Furstenberg multiple recurrence theorem \cite{Downey2019a} (see Furstenberg \cite{Furstenberg1981} for the classical theorem). The effective topological Poincar\'{e} recurrence theorem was introduced in Jindal \cite{Jindal2014}, but has not yet been established for all sets of effective first category. We resolve this issue, showing that the \emph{topological} Poincar\'{e} theorem holds effectively. Our proof uses the game-theoretic characterization of effective first category sets from the previous section. 

\begin{definition}[Recurrence]
	For a space $X\subseteq\mathbb{R}^n$ equipped with a homeomorphism $T$ onto itself, and and open set $G\subseteq X$, a point $x\in G$ is said to be \textit{recurrent with respect to $G$}, if $T^ix\in G$ for infinitely many $i\geq0$. $x$ is said to be \textit{recurrent under $T$}, if for every open $U\ni x$, $x$ is recurrent with respect to $U$. 
\end{definition}

The points which are not recurrent are said to be non-recurrent. 

\begin{definition}[Wandering Set]
	For a space $X$ equipped with a surjective map $T:X\to X$, an open set $E\subseteq X$ is said to be \emph{wandering} if the sets in the sequence $\{T^{-i}E\}_{i\geq0}$ are mutually disjoint.
\end{definition}

We now introduce the computability restrictions on the map required to establish our theorem. 

\begin{definition}[Computable Homeomorphism]
	For two effective $T_0$ spaces $(X_1,\tau_1,\beta_1,\nu_1)$ and $(X_2,\tau_2,\beta_2,\nu_2)$, a homeomorphism $f:(X_1,\tau_1,\beta_1,\nu_1)\to(X_2,\tau_2,\beta_2,\nu_2)$ is said	to be computable, if $\nu_2^{-1}\circ f\circ\nu_1$ is a total computable bijection mapping basic open sets to basic open sets, such that its inverse is also a total computable bijection. 
\end{definition} 

Observe that the image of a c.e. open set under a computable homeomorphism is also c.e. open. 

The classical version of the Poincar\'e recurrence theorem is as follows. 

\begin{theorem}[Poincar\'e Recurrence Theorem]\label{thm:classprt}
	For a bounded open region $X\subseteq\mathbb{R}^n$ equipped with a 	measure preserving homeomorphism $T$ onto itself, all the points of	$X$, except a set of measure zero and first category, are recurrent under $T$.
\end{theorem}

A more general topological recurrence theorem for a non-invertible map over Baire space is known \cite{Malicky2007}, but for the effective version, we restrict ourselves to the case of computable homeomorphisms over computable Euclidean space. We take $T$ to be a computable homeomorphism onto the space such that the images (and inverse images) of basic open sets are computable, and the space under $T$ \textit{does not admit any non-empty open wandering set}. This crucial assumption is part of the theorem statement and also remarked in the proof. 

We now characterize the set of non-recurrent points of the space by an effective Banach-Mazur game. We see that by definition of the set, we can come up with a winning strategy for one of the players. Then, by the results in \cref{sec:effbmgame}, we obtain the desired characterization. 

\begin{theorem}[Effective Poincar\'e Recurrence Theorem for Category]\label{thm:effcatprt}
	For a bounded c.e. open region $X\subseteq\mathbb{R}^n$, equipped with a computable homeomorphism $T$ onto itself, admitting no non-empty wandering open set, all the points of $X$, except a set of effective first category, are recurrent under $T$. 
\end{theorem}

We work with the basis $\eta = \{B(q,d) : d \in \Q, q \in \Q^n\}$ for $\R^n$. We fix a representation $\nu: \Sigma^* \dashrightarrow \beta$ such that $\nu\left(\langle e_1(x), e_2(d)\rangle\right) = B(x,d)$, where $e_2: \Sigma^* \to \Q^d$ and $e_2: \Sigma^* \to \Q$ are computable bijections, and $\langle,\rangle: \Sigma^* \times \Sigma^* \to \Sigma^*$ is a computable bijective encoding for pairs of strings. For inputs which are not of the above form, $\nu$ is undefined. Each rational has a computable \emph{name} $\{B(q,2^{-n}) : n \in\N\}$ where each element in the set is encoded using $\nu$. 

\begin{proof}
	Let $E\subsetneq X$ be a c.e. open set. Let $N(E)$ be the set defined as 
	\begin{align*}
		N(E) = \{x \in E :\left|\{j \in \N: T^{-j}x \in	E\}\right| < \infty\}. 
	\end{align*} 
	Then $N(E)$ is the set of non-recurrent points in $E$. Consider the effective Banach-Mazur game $BM\langle N(E), X\setminus N(E)\rangle$. We show that Player 2 has an effective winning	strategy in this game, establishing that $N(E)$ is a set of effective first category.
	
	Let $\mathcal{G}$ be the set of all basic open balls. Suppose, for any round $n \ge 1$, player 1 selects $G_{2n-1}\in\mathcal{G}$, a basic open set. We have $G_1 \supseteq G_2 \supseteq \dots \supseteq G_{2n-1}$. Consider the set 
	\begin{align*}
		F_n(E) = \{x \in E : (\forall j > n)(T^{-j}x \notin	E)\}. 
	\end{align*}
	Observe that $N(E)=\bigcup_{n\in\mathbb{N}}F_n(E)$. 
	
	We now show that every $F_n(E)$ is non-dense in $E$. Consider $F_k(E)$. Observe that $F_1(E)\supseteq F_2(E)\supseteq F_3(E)\supseteq\cdots$. Consider the set $H=\{x\in E:T^{-(k+1)}x\in E\}$. Clearly $H=E\cap T^{-(k+1)}(E)$. Since $E$ is a c.e. open set and $T$ is a computable homeomorphism, making $T^{-(k+1)}(E)$ c.e. open, $H$ is a c.e. open set. Now, by definition of $F_k(E)$, 
	\begin{align*}
		H\cap F_k(E)&=E\cap T^{-(k+1)}(E)\cap F_k(E)=\varnothing, 
	\end{align*}
	since $T^{-(k+1)}(E)\cap F_k(E)=\varnothing$. Also, $H\ne\varnothing$ since there are no non-empty open wandering sets. Hence, $H\cap(F_k(E))^c$ contains a non-empty c.e. open set. 
	
	Thus, the set $G_{2n-1} \setminus \overline{F_n(E)}$ contains a non-empty c.e. open set, uniformly in $n$. Let $G$ be the first basic open set in the computable enumeration of $G_{2n-1} \setminus \overline{F_n(E)}$. Player 2 plays $G_{2n}=G$. Then, by the definitions of the sets	$G_{2n}$, $n \ge 1$, no point in $N(E)$ can be present in $\bigcap_{n\in\N} G_n$. Thus, $N(E) \bigcap \left(\bigcap_{n\in\N} G_n\right) = \emptyset$. Hence the above strategy is an effective winning strategy for Player 2 in $BM\langle N(E),X\setminus N(E)\rangle$, establishing by \cref{thm:effbm1} that $N(E)$ is a set of effective first category. 
\end{proof}

%\section{Open Questions}
%We saw that given an oracle access to a c.e. player, the Banach-Mazur
% game can be won by that player, for arbitrary separable parent
% space. Can one draw a analogous conclusion for more general
% topological spaces?

\section*{Acknowledgments}
The authors thank anonymous referees for their valuable comments on an
earlier draft of this paper, and Alexander Melnikov for insightful
discussions.

\bibliography{main_jabref_shared}

@article{YMT99,
  title={Effective properties of sets and functions in metric spaces with computability structure},
  author={Yasugi, Mariko and Mori, Takakazu and Tsujii, Yoshiki},
  journal={Theoretical Computer Science},
  volume={219},
  number={1-2},
  pages={467--486},
  year={1999},
  publisher={Elsevier}
}

@InProceedings{Fenner1995,
  author    = {S. A. Fenner},
  booktitle = {Proceedings of the Tenth Structure in Complexity Theory Conference},
  title     = {Resource-bounded category: a stronger approach},
  year      = {1995},
  pages     = {182--192},
  publisher = {IEEE Computer Society Press},
}

@InProceedings{Lutz1987,
  author    = {J. H. Lutz},
  booktitle = {Proceedings of the Second Structure in Complexity Theory Conference},
  title     = {Resource-Bounded {Baire} Category and Small Circuits in Exponential Space},
  year      = {1987},
  pages     = {81--91},
}

@Book{Oxtoby1980,
  author    = {J. C. Oxtoby},
  publisher = {Springer-Verlag},
  title     = {Measure and Category},
  year      = {1980},
  address   = {Berlin},
  edition   = {second},
  doi       = {10.1007/978-1-4684-9339-9_22},
  issn      = {0072-5285},
  pages     = {86--91},
}

@Article{Vyugin1997,
  author  = {V. V. V'yugin},
  journal = {Theory of Probability and Its Applications},
  title   = {Effective convergence in probability and an ergodic theorem for individual random sequences},
  year    = {1997},
  number  = {1},
  pages   = {39--50},
  volume  = {42},
}

@Book{Furstenberg1981,
  author    = {Furstenberg, Hillel},
  publisher = {Princeton University Publishers},
  title     = {Recurrence in ergodic theory and combinatorial number theory},
  year      = {1981},
}

@InProceedings{Nandakumar2008,
  author    = {Nandakumar, S.},
  booktitle = {Proceedings of the 40th Annual Symposium on the Theory of Computing},
  title     = {An effective ergodic theorem and some applications},
  year      = {2008},
  pages     = {39--44},
}

@Article{Lutz1992,
  author  = {Lutz, J. H.},
  journal = {Journal of Computer and System Sciences},
  title   = {Almost Everywhere High Non-Uniform Complexity},
  year    = {1992},
  pages   = {220--258},
  volume  = {44},
}

@Article{Kano1993,
  author  = {Kano, H.},
  journal = {Osaka Journal of Mathematics},
  title   = {General Constructions of Normal numbers of the {K}orobov type},
  year    = {1993},
  pages   = {909--919},
  volume  = {4},
}

@Book{Walters1982,
  author    = {P. Walters},
  publisher = {Springer Verlag},
  title     = {An introduction to ergodic theory},
  year      = {1982},
  address   = {New York},
}

@InProceedings{Hoyrup2009b,
  author       = {Hoyrup, Mathieu and Rojas, Crist{\'o}bal},
  booktitle    = {International Colloquium on Automata, Languages, and Programming},
  title        = {Applications of effective probability theory to {M}artin-{L}{\"o}f randomness},
  year         = {2009},
  organization = {Springer},
  pages        = {549--561},
}

@Article{Downey2019a,
  author  = {Downey, Rodney G. and Nandakumar, Satyadev and Nies, Andr{\'e}},
  journal = {Notre Dame Journal of Formal Logic},
  title   = {Martin-L\"{o}f randomness implies multiple recurrence in effectively closed sets},
  year    = {2019},
  pages   = {491--502},
  volume  = {60},
  issue   = {3},
}

@Article{Nandakumar2016,
  author    = {Nandakumar, Satyadev and Vangapelli, Santhosh Kumar},
  journal   = {Theory of Computing Systems},
  title     = {Normality and finite-state dimension of {L}iouville
                  numbers},
  year      = 2016,
  number    = 3,
  pages     = {392--402},
  volume    = 58,
  publisher = {Springer US},
}

@MastersThesis{Jindal2014,
  author = {Pankaj Jindal},
  school = {Indian Institute of Technology Kanpur},
  title  = {Towards Proving that Poincare Non-Recurrent Points are Effective First Category Set},
  year   = {2014},
}

@MastersThesis{Vijayvargiya2014,
  author = {Nitest Vijayvargiya},
  school = {Indian Institute of Technology Kanpur},
  title  = {Poincare Non-Recurrent Points form an Effective Measure Zero Set},
  year   = {2014},
}

@Article{Grubba2007,
  author    = {Tanja Grubba and Matthias Schr{\"{o}}der and Klaus Weihrauch},
  journal   = {Math. Log. Q.},
  title     = {Computable metrization},
  year      = {2007},
  number    = {4-5},
  pages     = {381--395},
  volume    = {53},
  bibsource = {dblp computer science bibliography, https://dblp.org},
  doi       = {10.1002/MALQ.200710009},
}

@Book{Engelking1989,
  author     = {Engelking, Ryszard},
  publisher  = {Heldermann Verlag, Berlin},
  title      = {General topology},
  year       = {1989},
  edition    = {Second},
  isbn       = {3-88538-006-4},
  note       = {Translated from the Polish by the author},
  series     = {Sigma Series in Pure Mathematics},
  volume     = {6},
  mrclass    = {54-01 (54-02)},
  mrnumber   = {1039321},
  mrreviewer = {Gary\ Gruenhage},
  pages      = {viii+529},
}

@InProceedings{Brattka2001,
  author    = {Vasco Brattka},
  booktitle = {Mathematical Foundations of Computer Science 2001, 26th International Symposium, {MFCS} 2001 Marianske Lazne, Czech Republic, August 27-31, 2001, Proceedings},
  title     = {Computable Versions of Baire's Category Theorem},
  year      = {2001},
  editor    = {Jir{\'{\i}} Sgall and Ales Pultr and Petr Kolman},
  pages     = {224--235},
  publisher = {Springer},
  series    = {Lecture Notes in Computer Science},
  volume    = {2136},
  bibsource = {dblp computer science bibliography, https://dblp.org},
  doi       = {10.1007/3-540-44683-4\_20},
  url       = {https://doi.org/10.1007/3-540-44683-4_20},
}

@InCollection{Oxtoby1957,
  author    = {Oxtoby, John C.},
  booktitle = {Contributions to the theory of games, vol. 3},
  publisher = {Princeton Univ. Press, Princeton, NJ},
  title     = {The {B}anach-{M}azur game and {B}anach category theorem},
  year      = {1957},
  pages     = {159--163},
  series    = {Ann. of Math. Stud., no. 39},
  mrnumber  = {93741},
}

@Article{Malicky2007,
  author   = {Mali\v{c}k\'{y}, Peter},
  journal  = {Topology Appl.},
  title    = {Category version of the {P}oincar\'e{} recurrence theorem},
  year     = {2007},
  issn     = {0166-8641,1879-3207},
  number   = {14},
  pages    = {2709--2713},
  volume   = {154},
  doi      = {10.1016/j.topol.2007.05.004},
  fjournal = {Topology and its Applications},
  mrclass  = {37A99 (37B20 54H05 54H20)},
  mrnumber = {2340953},
  url      = {https://doi.org/10.1016/j.topol.2007.05.004},
}

@Article{Mycielski1956,
  author   = {Mycielski, Jan and Swierczkowski, S. and Zipolhkeba, A.},
  journal  = {Bull. Acad. Polon. Sci. Cl. III.},
  title    = {On infinite positional games},
  year     = {1956},
  pages    = {485--488},
  volume   = {4},
  mrnumber = {86725},
}

@Article{Banach1930,
  author  = {Banach, S.},
  journal = {Fund. Math.},
  title   = {Th\'{e}oreme sur les ensembles des premi\`{e}re cat\'{e}gorie},
  year    = {1930},
  pages   = {395-398},
  volume  = {16},
}

@Book{Kuratowski1948,
  author    = {Kuratowski},
  publisher = {Warsaw-Wroclaw},
  title     = {Topologie I},
  year      = {1948},
}

@Article{Hilmy1940,
  author     = {Hilmy, Heinrich},
  journal    = {Rec. Math. [Mat. Sbornik] N.S.},
  title      = {Sur la r\'ecurrence ergodique dans les syst\`emes dynamiques},
  year       = {1940},
  pages      = {101--109},
  volume     = {7/49},
  fjournal   = {Rec. Math. [Mat. Sbornik] N.S.},
  mrclass    = {46.3X},
  mrnumber   = {2029},
  mrreviewer = {Gustav\ A.\ Hedlund},
}

@Book{Poincare1899,
  author    = {Poincar\'{e}, H.},
  publisher = {Gauthier-Villars},
  title     = {Les m\'{e}thodes nouvelles de la m\'{e}chanique c\'{e}leste},
  year      = {1899},
  volume    = {3},
}

@Book{Birkhoff1927,
  author    = {Birkhoff, G. D.},
  publisher = {American Mathematical Society Colloq. Publ.},
  title     = {Dynamical Systems},
  year      = {1927},
  volume    = {9},
}

@Article{Pinsker1973,
  author  = {Pinsker, M.. A.},
  journal = {SIAM J. Computing},
  title   = {On the complexity of a conentrator},
  year    = {1973},
}

@Article{Zermelo1896,
  author  = {Zermelo, E.},
  journal = {Ann. Phys.},
  title   = {Uber einen Satz der Dynamik and die Mechanische W\"{a}rmetheorie.},
  year    = {1896},
  pages   = {485-494},
  volume  = {57},
}

@Article{Boltzmann1896,
  author  = {Boltzmann, L.},
  journal = {Ann. Phy.},
  title   = {\"{U}ber die mechanische Bedeutung des zweiten Hauptsatzes der W\"{a}rmetheorie},
  year    = {1896},
  number  = {773},
  volume  = {57},
}

@Book{Mauldin2015,
  editor    = {Mauldin, R. Daniel},
  publisher = {Birkh\"{a}user/Springer, Cham},
  title     = {The {S}cottish {B}ook},
  year      = {2015},
  edition   = {Second},
  isbn      = {978-3-319-22896-9; 978-3-319-22897-6},
  note      = {Mathematics from the Scottish Caf\'{e} with selected problems from the new Scottish Book, Including selected papers presented at the Scottish Book Conference held at North Texas University, Denton, TX, May 1979},
  doi       = {10.1007/978-3-319-22897-6},
  mrnumber  = {3242261},
  pages     = {xvii+322},
}

@article{Liouville1844,
  author   = {Liouville, Joseph},
  title    = {M{\'e}moires et communications},
  journal  = {Comptes rendus de l'Acad{\'e}mie des Sciences},
  year     = {1844},
  month    = {may},
  volume   = {18},
  number   = {20, 21},
  pages    = {883--885, 910--911},
  language = {French}
}

@article{Staiger2002,
title = {The Kolmogorov complexity of real numbers},
journal = {Theoretical Computer Science},
volume = {284},
number = {2},
pages = {455-466},
year = {2002},
issn = {0304-3975},
doi = {https://doi.org/10.1016/S0304-3975(01)00102-5},
url = {https://www.sciencedirect.com/science/article/pii/S0304397501001025},
author = {Ludwig Staiger}
}

@article{Lisagor1981,
  author  = {Lisagor, L. R.},
  title   = {The {Banach-Mazur} game},
  journal = {Mathematics of the USSR-Sbornik},
  year    = {1981},
  volume  = {38},
  number  = {2},
  pages   = {201--216},
  doi     = {10.1070/SM1981v038n02ABEH001229}
}

@inproceedings{BreutzmannJuedesLutz2001,
  author     = {Breutzmann, Josef M. and Juedes, David W. and Lutz, Jack H.},
  title      = {Baire Category and Nowhere Differentiability for Feasible Real Functions},
  booktitle  = {Algorithms and Computation. ISAAC 2001},
  series     = {Lecture Notes in Computer Science},
  volume     = {2223},
  publisher  = {Springer, Berlin, Heidelberg},
  year       = {2001},
  pages      = {219--230},
  doi        = {10.1007/3-540-45678-3_20}
}

\end{document}